\documentclass[12pt,english]{amsart}

\input xy
\xyoption{all}

\usepackage{comment}
\usepackage[latin1]{inputenc}
\usepackage{amsfonts,amsmath,amssymb,amsthm}

\newcommand{\Zz}{\mathbb{Z}}
\newcommand{\Qq}{\mathbb{Q}}

{\theoremstyle{plain}
\newtheorem{theorem}{Theorem}[section]    

\newtheorem{twisting lemma}[theorem]{Twisting lemma}

\newtheorem{lemma}[theorem]{Lemma}       
\newtheorem{proposition}[theorem]{Proposition}  

}
{\theoremstyle{remark}
\newtheorem{definition}[theorem]{Definition}      
\newtheorem{remark}[theorem]{Remark}   

}

\def\cm{\hbox{\hbox{\rm C}\kern-5pt{\raise 1pt\hbox{$|$}}}}

\def\lhfl#1#2{\smash{\mathop{\hbox to 12mm{\leftarrowfill}}
\limits^{#1}_{#2}}}

\def\rhfl#1#2{\smash{\mathop{\hbox to 12mm{\rightarrowfill}}
\limits^{#1}_{#2}}}

\def\build#1_#2^#3{\mathrel{
\mathop{\kern 0pt#1}\limits_{#2}^{#3}}}

\def\htrait#1#2{\smash{\mathop{\hbox to 12mm{\hrulefill}}
\limits^{#1}_{#2}}}

\def\sxbullet{{\raise 2pt\hbox{\bf .}}}
\numberwithin{equation}{section}

\begin{document}

\title{Hilbert specialization results with local conditions}

\author{Fran\c cois Legrand}

\email{flegrand@post.tau.ac.il}

\address{School of Mathematical Sciences, Tel Aviv University, Ramat Aviv, Tel Aviv 6997801, Israel}

\address{Department of Mathematics and Computer Science, the Open University of Israel, Ra'anana 4353701, Israel}

\date{\today}

\maketitle

\begin{abstract}
Given a field $k$ of characteristic zero and an indeterminate $T$, the main topic of the paper is the construction of specializations of any given finite extension of $k(T)$ of degree $n$ that are degree $n$ field extensions of $k$ with specified local behavior at any given finite set of primes of $k$. First, we give a full non-Galois analog of a result with a ramified type conclusion from a preceding paper and next we prove a unifying statement which combines our results and previous work devoted to the unramified part of the problem in the case $k$ is a number field.
\end{abstract}

\section{Introduction}

Given a field $k$ of characteristic zero, an indeterminate $T$, a finite extension $E/k(T)$ of degree $n$ and a point $t_0 \in \mathbb{P}^1(k)$, not a branch point, the {\it{specialization of $E/k(T)$ at $t_0$}} is a {\it{$k$-\'etale algebra of degree $n$}}, {\it{i.e.}} a finite product $\prod_l F_l/k$ of finite extensions of $k$ such that $\sum_l [F_l:k]=n$. See \S2 for basic terminology. For example, if $E/k(T)$ is given by a monic irreducible (in $Y$) polynomial $P(T,Y)\in k[T][Y]$, it is the product of extensions of $k$ corresponding to the irreducible factors of $P(t_0,Y)$ (for all but finitely many $t_0\in k$).

The main topic of the paper is the construction of specialization points $t_0 \in \mathbb{P}^1(k)$ such that the following two conditions are satisfied:

\noindent
- the specialization of $E/k(T)$ at $t_0$ consists of a single degree $n$ field extension $E_{t_0}/k$,

\noindent
- the extension $E_{t_0}/k$ has specified local behavior (ramified or unramified) at any given finite set of primes of $k$. By ``with specified local behavior", we mean, in the case the extension $E/k(T)$ is Galois, with specified inertia groups or Frobenius and, in the general case, with specified ramification indices or residue degrees.

The unramified part of this problem is studied in \cite{DL12} for arbitrary finite extensions of $k(T)$ whereas the ramified one is studied in \cite{Leg13a} in the particular case where the extension $E/k(T)$ is {\it{$k$-regular}} ({\it{i.e.}} $E \cap \overline{k} = k$) and Galois. The aim of this paper consists first in handling the ramified case for arbitrary finite extensions of $k(T)$ and next in providing unifying results.

\vspace{2.3mm}

\noindent
1.1. {\bf{The ramified part of the problem}} should be studied within some classical limitations that we recall briefly below. We refer to \S3 for precise statements (the ``Specialization Inertia Theorem" and Proposition \ref{transition2}), more details and references.

Let $k$ be the quotient field of a Dedekind domain $A$ of characteristic zero, $E/k(T)$ a finite extension of degree $n$, $\widehat{E}/k(T)$ its Galois closure and $\{t_1,\dots,t_r\} \subset \mathbb{P}^1(\overline{k})$ its branch point set. Then

\noindent
- a necessary condition for a given prime $\mathcal{P}$ of $A$ (not in a certain finite list $\mathcal{S}_{\rm{exc}}$ depending on $E/k(T)$) to ramify in a specialization of $E/k(T)$ requires $\mathcal{P}$ to be a {\it{prime divisor}} of the minimal polynomial $m_{i_\mathcal{P}}(T)$ over $k$ of some branch point $t_{i_\mathcal{P}}$ (unique up to $k$-conjugation), {\it{i.e.}} $m_{i_\mathcal{P}}(t_{0,\mathcal{P}})$ has positive $\mathcal{P}$-adic valuation for some point $t_{0,\mathcal{P}} \in k$,

\noindent
- the inertia group at $\mathcal{P}$ of the specialization at a given point $t_0 \in \mathbb{P}^1(k) \setminus \{t_1,\dots,t_r\}$ of $\widehat{E}/k(T)$ is generated by a power $g_{i_{\mathcal{P}}}^{a_\mathcal{P}}$ (depending on $t_0$ and $t_{i_\mathcal{P}}$) of the {\it{distinguished generator}} $g_{i_\mathcal{P}}$ of the inertia group of $\widehat{E}\overline{k}/\overline{k}(T)$ at some prime lying over $(T-t_{i_\mathcal{P}}) \overline{k}[T-t_{i_\mathcal{P}}]$,

\noindent
- the set of all ramification indices at $\mathcal{P}$ of the specialization at $t_0$ of $E/k(T)$ is the set of all lengths of disjoint cycles involved in the cycle decomposition in $S_d$ of the image of $g_{i_{\mathcal{P}}}^{a_\mathcal{P}}$ {\it{via}} the action $\nu$ of  ${\rm{Gal}}(\widehat{E}\overline{k}/\overline{k}(T))$ on all $\overline{k}(T)$-embeddings of ${E} \overline{k}$ in a given algebraic closure of $\overline{k}(T)$ (with $d=[E \overline{k} : \overline{k}(T)]$).

In \cite{Leg13a}, we provide some converse to the Galois conclusion. Let $\mathcal{S}$ be a finite set of prime ideals $\mathcal{P}$ of $A$ which are not in the list $\mathcal{S}_{\rm{exc}}$, each given with a couple $(i_\mathcal{P},a_\mathcal{P})$ where $i_\mathcal{P}$ is an index in $\{1,\dots,r\}$ such that $\mathcal{P}$ is a prime divisor of $m_{i_\mathcal{P}}(T)$ and $a_\mathcal{P} \geq 1$ is an integer. Then

\noindent
- \cite[Corollary 3.3]{Leg13a} shows that, if $\widehat{E}/k(T)$ is $k$-regular\footnote{Recall that this extension is not $k$-regular in general, even if $E/k(T)$ is $k$-regular.} and $k$ is hilbertian\footnote{{\it{e.g.}} $k$ is a number field or a finite extension of a rational function field $\kappa(X)$ with $\kappa$ an arbitrary field and $X$ an indeterminate. See {\it{e.g.}} \cite{FJ08} for more on hilbertian fields.}, then, for infinitely many points $t_0 \in \mathbb{P}^1(k) \setminus \{t_1,\dots,t_r\}$, the specialization of $\widehat{E}/k(T)$ at $t_0$ has Galois group ${\rm{Gal}}(\widehat{E}/k(T))$ and inertia group at $\mathcal{P}$ generated by $g_{i_\mathcal{P}}^{a_\mathcal{P}}$ ($\mathcal{P} \in \mathcal{S}$),

\noindent
- Theorem \ref{Hilbert} in \S4 of this paper provides a full non-Galois analog:

\vspace{2mm}

\noindent
{\bf{Theorem 1.}} {\it{Assume that $k$ is hilbertian and continue with the data from above. Then, for infinitely many points $t_0 \in \mathbb{P}^1(k) \setminus \{t_1,\dots,t_r\}$,

\vspace{0.5mm}

\noindent
{\rm{(1)}} the specialization of $E/k(T)$ at $t_0$ consists of a single degree $n$ field extension $E_{t_0}/k$,

\vspace{0.5mm}

\noindent
{\rm{(2)}} the set of all ramification indices of $E_{t_0}/k$ at each $\mathcal{P} \in \mathcal{S}$ is the set of all lengths of disjoint cycles involved in the decomposition of $\nu(g_{i_{\mathcal{P}}}^{a_\mathcal{P}})$.}}

\vspace{2mm}

\noindent
Under some g-completeness hypothesis, the hilbertianity assumption can be relaxed; see Theorem \ref{gc}. Furthermore, our results show that \cite[Corollary 3.3]{Leg13a} still holds if $\widehat{E}/k(T)$ is not $k$-regular.

\vspace{3mm}

\noindent
1.2. {\bf{The mixed situation.}} In \S5, we prove a unifying result, in the number field case\footnote{The situation of a base field which is a rational function field $\kappa(X)$ with coefficients in a field $\kappa$ with suitable arithmetic properties (and $X$ an indeterminate) can also be considered. We refer to \cite[\S4]{DG11} for more on this case.}, which combines some of our results from \S4 and the results from the already alluded to previous work devoted to the unramified part of the problem. Theorem \ref{DGL} gives our precise result.

Moreover, in the special case $k=\mathbb{Q}$, explicit bounds on the discriminant of our specializations can be added to our conclusions; see \S5.3.1. For instance, we obtain the following result; see \S5.3.2.

\vspace{3mm}

\noindent
{\bf{Theorem 2.}} {\it{Let $E/\mathbb{Q}(T)$ be a finite extension of degree $n$ with at least one $\mathbb{Q}$-rational branch point and such that the Galois closure $\widehat{E}/\mathbb{Q}(T)$ is $\Qq$-regular. Let $G$ be the Galois group of $\widehat{E}/\mathbb{Q}(T)$. Then there exist three positive constants $m_0$, $\alpha$ and $ \beta$ (depending only on $E/\mathbb{Q}(T)$) that satisfy the following property. Given two disjoint finite sets $\mathcal{S}_{\rm{ur}}$ and $\mathcal{S}_{\rm{ra}}$ of prime numbers $p \geq m_0$, there exist rational numbers $t_0$ such that

\vspace{0.5mm}

\noindent
{\rm{(1)}} the specialization of $E/\mathbb{Q}(T)$ at $t_{0}$ consists of a single degree $n$ field extension $E_{t_{0}}/\mathbb{Q}$ and the Galois closure $\widehat{E_{t_0}}/\mathbb{Q}$ of $E_{t_0}/\Qq$ has Galois group $G$,

\vspace{0.5mm}

\noindent
{\rm{(2)}} no prime number $p \in \mathcal{S}_{\rm{ur}}$ ramifies in $E_{t_{0}}/\mathbb{Q}$,

\vspace{0.5mm}

\noindent
{\rm{(3)}} each prime number $p \in \mathcal{S}_{\rm{ra}}$ ramifies in $E_{t_{0}}/\mathbb{Q}$,

\vspace{0.5mm}

\noindent
{\rm{(4)}} the discriminant $d_{E_{t_0}}$ of $E_{t_0}/\mathbb{Q}$ satisfies $$\prod_{p \in \mathcal{S}_{\rm{ra}}} p \leq |d_{E_{t_0}}| \leq \alpha \prod_{p \in \mathcal{S}_{\rm{ur}} \cup \mathcal{S}_{\rm{ra}}} p^\beta.$$}}

\noindent
In addition to many $\Qq$-regular Galois extensions of $\mathbb{Q}(T)$ with various Galois groups ({\it{e.g.}} abelian groups of even order, symmetric groups, alternating groups, some other non-abelian simple groups (including the Monster group), {\it{etc.}}), several non-Galois finite extensions of $\mathbb{Q}(T)$ satisfy the assumptions of Theorem 2. For instance, given an integer $n \geq 3$, this is true of the finite extension of $\mathbb{Q}(T)$ generated by one root of the irreducible trinomial $Y^n-Y-T$ (in which case $G=S_n$) \cite[\S4.4]{Ser92}. See also \cite[\S2.4]{Sch00} (or Remark 4.4) for another examples with $G=S_n$ ($n \geq 3$) and \cite[\S4.5]{Ser92} for examples with $G=A_n$ ($n \geq 5$).

\vspace{2.5mm}

{\bf{Acknowledgments.}} The author wishes to Pierre D\`ebes for helpful discussions as well as the anonymous referee for his/her thorough work and many valuable comments on a preliminary version of the paper. This work is partially supported by the Israel Science Foundation (grants No. 40/14 and No. 696/13).

\section{Basics on finite extensions of $k(T)$}

Given a field $k$ of characteristic zero, fix an algebraic closure $\overline{k}$ of $k$.

Recall that a {\it{finite $k$-\'etale algebra}} $\prod_l F_l/k$ is a finite product of finite field extensions $F_l/k$. The integer $\sum_l [F_l:k]$ is the {\it{degree of $\prod_l F_l/k$}}.

\subsection{Generalities} 
Let $T$ be an indeterminate. A finite degree $n$ extension $E/k(T)$ is {\it{$k$-regular}} if $E \cap \overline{k}=k$. In general, there is some {\it{constant extension in $E/k(T)$}}, which we denote by ${k_E}/k$ and define by ${k_E}= E \cap \overline{k}$. Note that the extension $E/k_E(T)$ is $k_E$-regular. The special case ${k_E}=k$ corresponds to the situation $E/k(T)$ is $k$-regular.

Denote the Galois closure of $E/k(T)$ by $\widehat{E}/k(T)$. The Galois group ${\rm{Gal}}(\widehat{E}/k(T))$ is denoted by $G$ and called {\it{the Galois group of $E/k(T)$}}. Next, denote by $\widehat{E} \overline{k}$ the {\it{compositum}} of $\widehat{E}$ and $\overline{k}(T)$ (in a fixed algebraic closure of $k(T)$). The Galois group ${\rm{Gal}}(\widehat{E} \overline{k}/ \overline{k}(T))$ is denoted by $\overline{G}$ and called {\it{the geometric Galois group of $E/k(T)$}}; it is a normal subgroup of $G$ (these two groups coincide if and only if $\widehat{E}/k(T)$ is $k$-regular).

{\it{Via}} its action on all $\overline{k}(T)$-embeddings of ${E} \overline{k}$ in a given algebraic clo-sure of $\overline{k}(T)$,  $\overline{G}$ may be viewed as a subgroup of the permutation group of all these embeddings. Up to some labeling of them, it may be viewed as a subgroup of $S_d$ (with $d=[E \overline{k} : \overline{k}(T)]$). Denote the corresponding {\hbox{morphism by $\nu : \overline{G} \rightarrow S_d$ and call it the {\it{embedding morphism of $\overline{G}$ in $S_d$}}.}}

\subsection{Branch points}

Given $t_0 \in \mathbb{P}^1(\overline{k})$, denote the integral closure of $\overline{k}[T-t_0]$ in $\widehat{E} \overline{k}$ by $\overline{B}$ \footnote{Replace $T-t_0$ by $1/T$ if $t_0=\infty$.}. We say that $t_0$ is {\it{a branch point of $E/k(T)$}} if the prime ideal $(T-t_0) \, \overline{k}[T-t_0]$ ramifies in $\overline{B}$. The extension $E/k(T)$ has only finitely many branch points; their number is positive if and only if $\widehat{E}\overline{k} \not=\overline{k}(T)$. From now on, we assume that this last condition holds and denote the branch points by $t_1,\dots,t_r$ ($r \geq1$).

\subsection{Inertia canonical invariants} 

For each positive integer $n$, fix a primitive $n$-th root of unity $\zeta_n$. Assume that the system $\{\zeta_n\}_n$ is {\it{coherent}}, {\it{i.e.}} $\zeta_{nm}^n=\zeta_m$ for any positive integers $n$ and $m$.

To each $t_i$ can be associated a conjugacy class $C_i$ of $\overline G$, called the {\it{inertia canonical conjugacy class (associated with $t_i$)}}, in the following way. The inertia groups of the prime ideals lying over $(T-t_i) \, \overline{k}[T-t_i]$ in the extension $\widehat{E}\overline{k}/\overline{k}(T)$ are cyclic conjugate groups of order equal to the ramification index $e_i$. Furthermore, each of them has a distinguished generator corresponding to the automorphism $(T-t_i)^{1/e_i} \mapsto \zeta_{e_i} (T-t_i)^{1/e_i}$ of $\overline{k}(((T-t_i)^{1/e_i}))$. Then $C_i$ is the conjugacy class of all the distinguished generators of the inertia groups of the prime ideals lying over $(T-t_i) \, \overline{k}[T-t_i]$ in the extension $\widehat{E}\overline{k}/\overline{k}(T)$. The unordered $r$-tuple $({C_1},\dots,{C_r})$ is called {\it{the inertia canonical invariant of $\widehat{E}/k(T)$}}.

Denote by $C_i^{S_d}$ the conjugacy class of $S_d$ corresponding to $C_i$ {\it{via}} the embedding morphism $\nu : \overline{G} \rightarrow S_d$ ($i \in \{1,\dots,r\}$). The unordered $r$-tuple $(C_1^{S_d},\dots,C_r^{S_d})$ is called {\it{the inertia canonical invariant of $E/k(T)$}}.

\subsection{Specializations} Let $t_0 \in \mathbb{P}^1(k) \setminus \{t_1,\dots,t_r\}$.

\subsubsection{Galois case} The residue field at some prime ideal lying over $(T-t_0)k[T-t_0]$ in the extension $\widehat{E}/k(T)$ is denoted by $\widehat{E}_{t_0}$ and we call the extension $\widehat{E}_{t_0}/k$ {\it{the specialization of $\widehat{E}/k(T)$ at $t_0$}}. This does not depend on the choice of the prime ideal lying over $(T-t_0)k[T-t_0]$ since the extension $\widehat{E}/k(T)$ is Galois. The specialization of $\widehat{E}/k(T)$ at $t_0$ is a Galois extension of $k$ whose Galois group is a subgroup of $G$, namely the decomposition group of $\widehat{E}/k(T)$ at a prime ideal lying over $(T-t_0)k[T-t_0]$.

\subsubsection{General case} Denote the prime ideals lying over $(T-t_0)k[T-t_0]$ in the extension $E/k(T)$ by $\mathcal{P}_1, \dots, \mathcal{P}_s$. For each $l \in \{1,\dots,s\}$, the residue field at $\mathcal{P}_l$ is denoted by $E_{t_0,l}$ and the extension $E_{t_0,l}/k$ is called {\it{a specialization of $E/k(T)$ at $t_0$}}. The degree $n$ $k$-\'etale algebra $\prod_{l=1}^s E_{t_0,l}/k$ is called {\it{the specialization algebra of $E/k(T)$ at $t_0$}}. The {\it{compositum}} in $\overline k$ of the Galois closures of all specializations of $E/k(T)$ at $t_0$ is the specialization of the Galois closure $\widehat{E}/k(T)$ at $t_0$.

\vspace{2mm}

In the case the extension $E/k(T)$ is given by a polynomial $P(T,Y) \in k[T][Y]$, the following lemma is useful:

\begin{lemma} \label{spec non gal}
Let $P(T,Y) \in k[T][Y]$ be a monic (in $Y$) polynomial which is irreducible over $k(T)$ and such that $E$ is the field generated over $k(T)$ by one of its roots. Let $t_0 \in k$. Assume that $P(t_0,Y)$ is separable. Then the following two conditions hold.

\vspace{0.5mm}

\noindent
{\rm{(1)}} The point $t_0$ is not a branch point of $E/k(T)$.

\vspace{0.5mm}

\noindent
{\rm{(2)}} Consider the factorization $P(t_0,Y)=P_1(Y) \dots P_s(Y)$ of $P(t_0,Y)$ in irreducible polynomials $P_l(Y) \in k[Y]$ and denote the field generated over $k$ by one of the roots of $P_l(Y)$ by $F_l$ ($l \in \{1,\dots,s\}$). Then the specialization algebra of $E/k(T)$ at $t_0$ is the $k$-\'etale algebra $\prod_{l=1}^s F_l/k$.
\end{lemma}

\subsection{Notation}

The notation below will be used throughout the paper.

Let $A$ be a Dedekind domain of characteristic zero and $k$ its quotient field. We denote the Galois closure of a finite extension $F/k$ by $\widehat{F}/k$.

The minimal polynomial over $k$ of a given point $t \in \mathbb{P}^1(\overline{k})$ \footnote{Identify $\mathbb{P}^1(\overline{k})$ with $\overline{k} \cup \{\infty\}$.} is denoted by $m_{t}(T)$ (set $m_{t}(T)=1$ if $t = \infty$).

Let $E/k(T)$ be a finite extension of degree $n$ and $\widehat{E}/k(T)$ the Galois closure. Assume that $\widehat{E} \overline{k} \not=\overline{k}(T)$. Denote the branch point set by $\{t_1,\dots,t_r\}$, the constant extension in $\widehat{E}/k(T)$ by $k_{\widehat{E}}/k$, the Galois group ${\rm{Gal}}(\widehat{E}/k(T))$ by $G$, the geometric Galois group ${\rm{Gal}}(\widehat{E}\overline{k}/\overline{k}(T))$ by $\overline G$, the inertia canonical invariant of $\widehat{E}/k(T)$ by $({C_1}, \dots, {C_r})$, the degree of $E\overline{k}/\overline{k}(T)$ by $d$, the embedding morphism of $\overline{G}$ in $S_d$ by $\nu:\overline{G} \rightarrow S_d$ and the inertia canonical invariant of $E/k(T)$ by $(C_1^{S_d},\dots,C_r^{S_d})$. Finally, set $$m_{\bf{\underline{t}}}(T) = \prod_{i=1}^r m_{t_i}(T)$$ and, with $1/\infty = 0$ and $1 / 0 = \infty$, set $$m_{1/\bf{\underline{t}}}(T) = \prod_{i=1}^r m_{1/t_i}(T).$$

\section{General statements on ramification in specializations}

Below we complement some standard facts on the ramification in the specializations of the extension $E/k(T)$. Our goal is the Specialization Inertia Theorem which is a more precise version of two results of Beckmann \cite[Proposition 4.2 and Theorem 5.1]{Bec91}; see \S3.2.

\subsection{Preliminaries}

We need to recall some standard definitions.

Let $\mathcal{P}$ be a non-zero prime ideal of $A$. Denote the localization of $A$ at $\mathcal{P}$ by $A_\mathcal{P}$ and the valuation of $k$ corresponding to $\mathcal{P}$ by $v_\mathcal{P}$. We say that $\mathcal{P}$ ``unitizes" a point $t \in \mathbb{P}^1(\overline{k})$ if $t$ and $1/t$ are integral over $A_\mathcal{P}$, {\it{i.e.}} if $m_t(T)$ and $m_{1/t}(T)$ both have coefficients in $A_\mathcal{P}$.

Given a finite extension $F/k$ and a prime $\mathcal{P}_F$ of $F$ lying over $\mathcal{P}$, we denote the associated valuation of $F$ by $v_{\mathcal{P}_F}$.

\begin{definition} \label{rencontre} 
Given $t_0, t_1$ $\in \mathbb{P}^1(\overline{k})$, we say that {\it{$t_0$ and $t_1$ meet modulo $\mathcal{P}$}} if there exist a finite extension $F/k$ and a prime $\mathcal{P}_F$ of $F$ lying over $\mathcal{P}$ such that $t_0$, $t_1$ $\in \mathbb{P}^1(F)$ and one of these two conditions holds:

\vspace{0.5mm}

\noindent
(1) $v_{\mathcal{P}_F}(t_0) \geq 0$, $v_{\mathcal{P}_F}(t_1) \geq 0$ and $v_{\mathcal{P}_F}(t_0-t_1) > 0$,

\vspace{0.5mm}

\noindent
(2) $v_{\mathcal{P}_F}(t_0) \leq 0$, $v_{\mathcal{P}_F}(t_1) \leq 0$ and $v_{\mathcal{P}_F}((1/t_0) - (1/t_1)) > 0$ \footnote{Set $v_\mathcal{P}(\infty) = -\infty$ and $v_\mathcal{P}(0) = \infty$.}.
\end{definition}

\noindent
Note that Definition \ref{rencontre} does not depend on the choice of the finite extension $F/k$ such that $t_0, t_1 \in \mathbb{P}^1(F)$.

Let $t_1 \in \mathbb{P}^1(\overline{k})$. Assume that the constant coefficient $a_{t_1}$ of $m_{t_1}(T)$ satisfies $v_\mathcal{P}(a_{t_1})=0$ in the case $t_1 \not=0$ to make the intersection multiplicity well-defined in Definition \ref{**} below. Let $t_0 \in \mathbb{P}^1(k)$.

\begin{definition} \label{**}
{\it{The intersection multiplicity $I_{\mathcal{P}}(t_0,t_1)$ of $t_0$ and $t_1$ at $\mathcal{P}$}} is 
$I_\mathcal{P}(t_0,t_1)= \left \{ \begin{array} {ccc}
          v_{\mathcal{P}}(m_{t_1}(t_0)) & {\rm{if}} & v_\mathcal{P}(t_0) \geq 0, \\
          v_{\mathcal{P}}(m_{1/t_1}(1/t_0)) & {\rm{if}} &  v_\mathcal{P}(t_0) \leq 0. \\   
          \end{array} \right.$
\end{definition}

The following lemma, which is \cite[Lemma 2.5]{Leg13a}, will be used on several occasions in this paper.

\begin{lemma} \label{3.3}
{\rm{(1)}} If $I_\mathcal{P}(t_0,t_1) >0$, then $t_0$ and $t_1$ meet modulo $\mathcal{P}$.

\vspace{0.5mm}

\noindent
{\rm{(2)}} The converse holds if $\mathcal{P}$ unitizes $t_1$.
\end{lemma}

\begin{definition} \label{bon premier}
We say that the prime ideal $\mathcal{P}$ of $A$ is {\it{a bad prime for $E/k(T)$}} if $\mathcal{P}$ is one of the finitely many prime ideals of $A$ satisfying at least one of the following conditions:

\vspace{0.25mm}

\noindent
(1) $|\overline G| \in \mathcal{P}$,

\vspace{0.25mm}

\noindent
(2) two distinct branch points meet modulo $\mathcal{P}$,

\noindent
(3) $\widehat E/k(T)$ has {\it{vertical ramification}} at $\mathcal{P}$, {\it{i.e.}} the prime $\mathcal{P} A[T]$ of $A[T]$ ramifies in the integral closure of $A[T]$ in $\widehat E $,

\vspace{0.25mm}

\noindent
(4) $\mathcal{P}$ ramifies in the finite extension $k_{\widehat{E}}(t_1,\dots,t_r)/k$.

\vspace{0.25mm}

\noindent
Otherwise $\mathcal{P}$ is called {\it{a good prime for $E/k(T)$}}.
\end{definition}

Another goal of this section is Proposition \ref{transition2} (see \S3.3) which provides an explicit characterization of all primes of $k$ which ramify in a specialization of $E/k(T)$. We will need the following definition.

\begin{definition}
We say that the prime ideal $\mathcal{P}$ of $A$ is {\it{a prime divisor of a polynomial $P(T) \in k[T] \setminus k$}} if $v_{\mathcal{P}}(P(t_0)) >0$ for some $t_0 \in k$. 
\end{definition}

\subsection{Statement of the Specialization Inertia Theorem}

\noindent
\vspace{1.9mm}

\noindent
{\bf{Specialization Inertia Theorem.}} {\it{Let $t_0 \in \mathbb{P}^1(k) \setminus \{t_1,\dots,t_r\}$.

\vspace{0.5mm}

\noindent
{\rm{(1)}} If $\mathcal{P}$ ramifies in a specialization of $E/k(T)$ at $t_0$, then $\widehat E /k(T)$ has vertical ramification at $\mathcal{P}$ or $\mathcal{P}$ ramifies in the constant extension $k_{\widehat{E}}/k$ or $t_0$ meets some branch point modulo $\mathcal{P}$.

\vspace{0.5mm}

\noindent
{\rm{(2)}} Let $j \in \{1,\dots,r\}$. Assume that $\mathcal{P}$ is a good prime for ${E}/k(T)$ unitizing $t_j$ and $t_0, t_j$ meet modulo $\mathcal{P}$ \footnote{If $\mathcal{P}$ does not satisfy condition (2) of Definition \ref{bon premier}, then there is at most one index $j \in \{1,\dots,r\}$ (up to $k$-conjugation) such that $t_0$ and $t_j$ meet modulo $\mathcal{P}$.}. Then the following holds.

{\rm{(a)}} {\hbox{The inertia group at $\mathcal{P}$ of the specialization $\widehat{E}_{t_0}/k$ of $\widehat{E}/k(T)$ at}} 

\hspace{1mm}$t_0$ is generated by some element of $C_j^{I_\mathcal{P}(t_0,t_j)}(=\{g_j^{I_\mathcal{P}(t_0,t_j)} : g_j \in C_j \})$.

{\rm{(b)}} Assume that ${\rm{Gal}}(\widehat{E}_{t_0}/k_{\widehat{E}})= \overline{G}$. Then the set of all ramification

indices at $\mathcal{P}$ of a given specialization $E_{t_0,l}/k$ of $E/k(T)$ at $t_0$ is the

set of all lengths of disjoint cycles involved in the cycle decomposition

in $S_d$ of any element of $(C_j^{S_d})^{I_\mathcal{P}(t_0,t_j)} $.}}

\vspace{2mm}

\noindent
As already said, the result is a more precise version of two previous results of Beckmann. As a few gaps seem to appear in the original proofs, we have added some extra assumptions above. For the convenience of the reader, we offer a corrected proof in \S3.4, based on a version of the Specialization Inertia Theorem for $k$-regular Galois extensions of $k(T)$ from \cite[\S2.2.3]{Leg13a} (recalled as Theorem \ref{SIT 1} in \S3.4).

\subsection{Statement of Proposition \ref{transition2}}

\begin{proposition} \label{transition2}
Assume that $\mathcal{P}$ is a good prime for $E/k(T)$ unitizing each branch point. Then the following two conditions are equivalent:

\vspace{0.5mm}

\noindent
{\rm{(1)}} $\mathcal{P}$ ramifies in a specialization of $E/k(T)$,

\vspace{0.5mm}

\noindent
{\rm{(2)}} $\mathcal{P}$ is a prime divisor of $m_{\bf{\underline{t}}}(T) \cdot m_{1/\bf{\underline{t}}}(T)$.
\end{proposition}

\begin{proof}
The proof rests on the proof of \cite[Corollary 2.12]{Leg13a}, which is Proposition \ref{transition2} for $k$-regular Galois extensions of $k(T)$. We reproduce it below with the necessary adjustments for the bigger generality.

First, assume that $\mathcal{P}$ ramifies in a specialization of $E/k(T)$ at $t_0$ for some $t_0 \in \mathbb{P}^1(k)$. Suppose $v_\mathcal{P}(t_0) \geq 0$ (the other case for which $v_\mathcal{P}(t_0) <0$ is similar). By part (1) of the Specialization Inertia Theorem and as $\mathcal{P}$ is a good prime for $E/k(T)$, $t_0$ meets some branch point $t_i$ modulo $\mathcal{P}$. As $\mathcal{P}$ unitizes $t_i$, we may apply Lemma \ref{3.3} to get $v_\mathcal{P}(m_{t_i}(t_0)) >0$ (as $v_\mathcal{P}(t_0) \geq 0$). But $m_{t_1}(T), \dots,m_{t_r}(T), m_{1/t_1}(T), \dots,$ $m_{1/t_r}(T) \in A_\mathcal{P}[T]$ and $t_0 \in A_\mathcal{P}$. Then $v_\mathcal{P}(m_{\bf{\underline{t}}}(t_0) \cdot m_{1/\bf{\underline{t}}}(t_0)) >0$, as needed for (2).

Now, assume that condition (2) holds. Fix $t_0 \in k$ such that $v_\mathcal{P}(m_{\bf{\underline{t}}}(t_0) \cdot m_{1/\bf{\underline{t}}}(t_0)) >0$. As $m_{\bf{\underline{t}}}(T) \cdot m_{1/\bf{\underline{t}}}(T)$ has coefficients in $A_\mathcal{P}$ and is monic, one has $v_\mathcal{P}(t_0) \geq 0$. Assume that $v_\mathcal{P}(m_{\bf{\underline{t}}}(t_0)) >0$ (the other case for which $v_\mathcal{P}(m_{1/\bf{\underline{t}}}(t_0)) >0$ is similar). Then one has $v_\mathcal{P}(m_{t_i}(t_0)) >0$ for some $i \in \{1,\dots,r\}$. Let $x_\mathcal{P}$ be a generator of the maximal ideal $\mathcal{P} A_\mathcal{P}$ of $A_\mathcal{P}$. We claim that $v_\mathcal{P}(m_{t_i}(t_0)) =1$ or $v_\mathcal{P}(m_{t_i}(t_0 + x_\mathcal{P})) =1$, {\it{i.e.}} $I_{\mathcal{P}}(t_0,t_i) =1$ or $I_{\mathcal{P}}(t_0 + x_\mathcal{P},t_i) =1$ (as $v_\mathcal{P}(t_0) \geq 0$). Indeed, if $v_\mathcal{P}(m_{t_i}(t_0)) =1$, we are done. Then we may assume 
\begin{equation} \label{1}
v_\mathcal{P}(m_{t_i}(t_0)) \geq 2.
\end{equation}
By the Taylor formula, one has
\begin{equation} \label{2}
m_{t_i}(t_0 + x_\mathcal{P}) = m_{t_i}(t_0) + x_\mathcal{P} \cdot m'_{t_i}(t_0) + x_\mathcal{P}^2 \cdot R_\mathcal{P}
\end{equation}
for some $R_\mathcal{P} \in A_\mathcal{P}$ and, by \cite[Lemma 2.8]{Leg13a}, one has
\begin{equation} \label{3}
v_\mathcal{P}(m'_{t_i}(t_0))=0.
\end{equation}
Combining \eqref{1}, \eqref{2} and \eqref{3} provides $v_\mathcal{P}(m_{t_i}(t_0 + x_\mathcal{P}))=1$, thus proving our claim. Then we may apply Lemma \ref{3.3} to get that either $t_0$ and $t_i$ meet modulo $\mathcal{P}$ or $t_0 + x_\mathcal{P}$ and $t_i$ meet modulo $\mathcal{P}$. Moreover, we may assume that neither $t_0$ nor $t_0 + x_\mathcal{P}$ is a branch point. Then we may apply part (2)(a) of the Specialization Inertia Theorem to get that the inertia group at $\mathcal{P}$ of $\widehat{E}_{t_0}/k$ or of $\widehat{E}_{t_0+x_\mathcal{P}}/k$ is generated by an element of $C_i$. In particular, $\mathcal{P}$ ramifies in $\widehat{E}_{t_0}/k$ or in $\widehat{E}_{t_0 +x _\mathcal{P}}/k$. As a prime ideal of $A$ ramifies in the {\it{compositum}} of finitely many extensions of $k$ if and only if it ramifies in at least one of them (Abhyankar's lemma), we get that $\mathcal{P}$ ramifies in a specialization of $E/k(T)$ at $t_0$ or in a specialization of $E/k(T)$ at $t_0+x_\mathcal{P}$, as needed for condition (1).
\end{proof}

\subsection{A proof of the Specialization Inertia Theorem}

First of all, let us state Theorem \ref{SIT 1} below which is the Specialization Inertia Theorem for $k$-regular Galois extensions of $k(T)$.

\begin{theorem} \label{SIT 1}
Assume that the extension $E/k(T)$ is $k$-regular and Galois. Let $t_0 \in \mathbb{P}^1(k) \setminus \{t_1,\dots,t_r\}$.

\vspace{0.5mm}

\noindent
{\rm{(1)}} If $\mathcal{P}$ ramifies in $E_{t_0}/k$, then either $E /k(T)$ has vertical ramification at $\mathcal{P}$ or $t_0$ meets some branch point modulo $\mathcal{P}$.

\vspace{0.5mm}

\noindent
{\rm{(2)}} Let $j \in \{1,\dots,r\}$. Assume that $\mathcal{P}$ is a good prime for ${E}/k(T)$ unitizing $t_j$ and $t_0, t_j$ meet modulo $\mathcal{P}$. Then the inertia group at $\mathcal{P}$ of $E_{t_0}/k$ is generated by an element of $C_j^{I_\mathcal{P}(t_0,t_j)}$.
\end{theorem}

\noindent
As already said in \cite[\S2.2.3]{Leg13a}, Theorem \ref{SIT 1} is a version of \cite[Proposition 4.2]{Bec91} with less restrictive hypotheses. However, we have added condition (4) in Definition \ref{bon premier}, which here requires the prime ideal $\mathcal{P}$ not to ramify in the extension $k(t_1,\dots,t_r)/k$, to close a gap in the original proof. We refer to \cite[Remark 2.7]{Leg13a} for more details and \cite[\S1.2.1.4]{Leg13c} for a proof of Theorem \ref{SIT 1}.

\vspace{2mm}

Now, we explain how to obtain the general version of the Specialization Inertia Theorem from Theorem \ref{SIT 1}. In particular, we make use of our extra assumption ${\rm{Gal}}(\widehat{E}_{t_0}/k_{\widehat{E}})= \overline{G}$ in part (2)(b) to close some gaps in the proof of \cite[Theorem 5.1]{Bec91}.

\subsubsection{Proof of part (1)}

Assume that $\mathcal{P}$ ramifies in a specialization of $E/k(T)$ at $t_0$. Then $\mathcal{P}$ ramifies in the specialization $\widehat{E}_{t_0}/k$ of the Galois closure $\widehat{E}/k(T)$ at $t_0$. If $\mathcal{P}$ ramifies in the subextension $k_{\widehat{E}}/k$, we are done. Then we may assume that $\mathcal{P}$ does not ramify in $k_{\widehat{E}}/k$. Let $\mathcal{Q}$ be a prime of $k_{\widehat{E}}$ lying over $\mathcal{P}$. Then $\mathcal{Q}$ ramifies in the extension $\widehat{E}_{t_0}/k_{\widehat{E}}$, which is the specialization of $\widehat{E}/k_{\widehat{E}}(T)$ at $t_0$. As the extension $\widehat{E}/k_{\widehat{E}}(T)$ is $k_{\widehat{E}}$-regular and Galois, we may apply part (1) of Theorem \ref{SIT 1} to get that $\widehat{E}/k_{\widehat{E}}(T)$ has vertical ramification at $\mathcal{Q}$ or $t_0$ meets some branch point modulo $\mathcal{Q}$. First, assume that $\widehat{E}/k_{\widehat{E}}(T)$ has vertical ramification at $\mathcal{Q}$. As the extension $k_{\widehat{E}}/k$ is Galois and as $\mathcal{Q}$ lies over $\mathcal{P}$, we may use \cite[Lemma 2.1]{Bec91} to get that $\widehat{E}/k(T)$ has vertical ramification at $\mathcal{P}$. Now, assume that $t_0$ meets some branch point $t_i$ modulo $\mathcal{Q}$. Then $t_0$ meets $t_i$ modulo $\mathcal{P}$ (as $\mathcal{Q}$ lies over $\mathcal{P}$). Hence part (1) holds.

\subsubsection{Proof of part (2)(a)}

For simplicity, denote the field $k_{\widehat{E}}(t_1,\dots,t_r)$ by $F$. As $t_0$ and $t_j$ meet modulo $\mathcal{P}$, there exists a prime $\mathcal{Q}$ of $F$ lying over $\mathcal{P}$ such that $t_0$ and $t_j$ meet modulo $\mathcal{Q}$. Since $\mathcal{P}$ is a good prime for $\widehat{E}/k(T)$, the prime $\mathcal{Q}$ is a good prime for $\widehat{E}F/F(T)$ (use \cite[Lemma 2.1]{Bec91} and the fact that the extension $F/k$ is Galois to handle vertical ramification). Moreover, $\mathcal{Q}$ unitizes $t_j$ (as $\mathcal{P}$ does). As the exension $\widehat{E}F/F(T)$ is $F$-regular and Galois, we may apply part (2) of Theorem \ref{SIT 1} to get that the inertia group at $\mathcal{Q}$ of the specialization of $\widehat{E}F/F(T)$ at $t_0$ is generated by an element of $C_j^{I_\mathcal{Q}(t_0,t_j)}$. Since the extension $F/k$ is Galois and $\mathcal{P}$ does not ramify in $F/k$ (condition (4) of Definition {\ref{bon premier}), we may use \cite[Lemma 3.2]{Bec91} to get that the inertia group at $\mathcal{P}$ of the specialization of $\widehat{E}/k(T)$ at $t_0$ is the inertia group at $\mathcal{Q}$ of the specialization of $\widehat{E}F/F(T)$ at $t_0$. Then this inertia group at $\mathcal{P}$ is generated by an element of $C_j^{I_\mathcal{Q}(t_0,t_j)}$. Hence, to get part (2)(a), it suffices to prove $$I_\mathcal{Q}(t_0,t_j) = I_\mathcal{P}(t_0,t_j).$$

Assume that $v_\mathcal{P}(t_0) \geq 0$ (the other case for which $v_\mathcal{P}(t_0) <0$ is similar). Then one has 
\begin{equation} \label{11}
I_\mathcal{Q}(t_0,t_j) = v_\mathcal{Q}(t_0-t_j)
\end{equation}
(as $t_j$ is $F$-rational) and 
\begin{equation} \label{12}
I_\mathcal{P}(t_0,t_j) = v_\mathcal{P}(m_{t_j}(t_0)).
\end{equation}
As $\mathcal{P}$ does not ramify in the extension $F/k$, \eqref{12} provides
\begin{equation} \label{13}
I_\mathcal{P}(t_0,t_j) = v_\mathcal{Q}(m_{t_j}(t_0)).
\end{equation}
Let $t_{j'}$ be a $k$-conjugate of $t_j$ distinct from $t_j$. We claim that 
\begin{equation} \label{8}
v_\mathcal{Q}(t_0 - t_{j'}) = 0.
\end{equation}
 Indeed, note that $v_\mathcal{Q}(t_{j'})=0$ (as $\mathcal{P}$ unitizes $t_j$). Then we get 
\begin{equation} \label{14}
v_\mathcal{Q}(t_0-t_{j'}) \geq 0.
\end{equation}
Assume that $v_\mathcal{Q}(t_0-t_{j'}) \not= 0$. Then \eqref{14} provides
\begin{equation}\label{9}
v_\mathcal{Q}(t_0-t_{j'}) > 0.
\end{equation}
Moreover, as $t_0$ and $t_j$ meet modulo $\mathcal{Q}$ and as $\mathcal{Q}$ unitizes $t_j$, we may apply Lemma \ref{3.3} and use \eqref{11} to get
\begin{equation} \label{10}
 v_\mathcal{Q}(t_0-t_j)>0.
\end{equation}
Combining \eqref{9} and \eqref{10} provides $v_\mathcal{Q}(t_j - t_{j'}) >0.$ Then the distinct branch points $t_j$ and $t_{j'}$ meet modulo $\mathcal{Q}$, which cannot happen. Hence \eqref{8} holds. It then remains to combine \eqref{13}, \eqref{8} and \eqref{11} to get
$$I_\mathcal{P}(t_0,t_j) = v_\mathcal{Q}(m_{t_j}(t_0)) = v_\mathcal{Q}(t_0-t_j) = I_\mathcal{Q}(t_0,t_j),$$
as needed for part (2)(a).

\begin{remark} \label{proof}
With the previous notation, consider the restriction $\mathcal{P}'$ of $\mathcal{Q}$ to $k_{\widehat{E}}$. Then $t_0$ and $t_j$ meet modulo $\mathcal{P}'$ and one shows similarly that one has $I_\mathcal{Q}(t_0,t_j) = I_{\mathcal{P}'}(t_0,t_j).$ We then get the following statement that will be used on several occasions in the rest of this paper.

\vspace{1mm}

\noindent
{\it{Let $\mathcal{P}$ be a non-zero prime ideal of $A$, $t_0 \in \mathbb{P}^1(k) \setminus \{t_1,\dots,t_r\}$ and $j \in \{1, \dots,r\}$. Assume that $t_0$ and $t_j$ meet modulo $\mathcal{P}$ and $\mathcal{P}$ is a good prime for $E/k(T)$ unitizing $t_j$. Then there exists a prime $\mathcal{Q}$ of $k_{\widehat{E}}$ lying over $\mathcal{P}$ such that $t_0$ and $t_j$ meet modulo $\mathcal{Q}$ and $I_\mathcal{P}(t_0,t_j) = I_{\mathcal{Q}}(t_0,t_j).$}}
\end{remark}

\subsubsection{Proof of part (2)(b)}

Let $E_{t_0,l}/k$ be a specialization of $E/k(T)$ at $t_0$. From our extra assumption 
\begin{equation} \label{15}
{\rm{Gal}}(\widehat{E}_{t_0}/k_{\widehat{E}})= \overline{G},
\end{equation}
the specialization algebra of $Ek_{\widehat{E}}/k_{\widehat{E}}(T)$ at $t_0$ consists of a single degree $d$ extension $(Ek_{\widehat{E}})_{t_0}/k_{\widehat{E}}$. As $t_0$ and $t_j$ meet modulo $\mathcal{P}$ and $\mathcal{P}$ is a good prime for $E/k(T)$ unitizing $t_j$, we may apply Remark \ref{proof}. There exists a prime $\mathcal{Q}$ of $k_{\widehat{E}}$ lying over $\mathcal{P}$ such that $t_0$ and $t_j$ meet modulo $\mathcal{Q}$ and 
\begin{equation} \label{18}
I_\mathcal{P}(t_0,t_j) = I_{\mathcal{Q}}(t_0,t_j).
\end{equation}
Since $\mathcal{P}$ does not ramify in the constant extension $k_{\widehat{E}}/k$ (condition (4) of Definition \ref{bon premier}), we may apply \cite[Lemma 5.4]{Bec91} to get that the set of all ramification indices at $\mathcal{P}$ of $E_{t_0,l}/k$ is the set of all ramification indices at $\mathcal{Q}$ of $(Ek_{\widehat{E}})_{t_0}/k_{\widehat{E}}$. Then, by \eqref{18}, it suffices to show that the set of all ramification indices at $\mathcal{Q}$ of $(Ek_{\widehat{E}})_{t_0}/k_{\widehat{E}}$ is the set of all lengths of disjoint cycles involved in the cycle decomposition in $S_d$ of any element of $(C_j^{S_d})^{I_\mathcal{Q}(t_0,t_j)}$.

Let $\mathcal{Q}'$ be a prime lying over $\mathcal{Q}$ in $(Ek_{\widehat{E}})_{t_0}/k_{\widehat{E}}$. Below we determine the ramification index $e_{\mathcal{Q}'/\mathcal{Q}}$ of $\mathcal{Q}'$ over $\mathcal{Q}$. Let $\mathcal{Q}''$ be a prime lying over $\mathcal{Q}'$ in $\widehat{E}_{t_0}/(Ek_{\widehat{E}})_{t_0}$ and let $g_j \in C_j$. By part (2)(a) of the Specialization Inertia Theorem (or by part (2) of Theorem \ref{SIT 1}), the inertia group of $\widehat{E}_{t_0}/k_{\widehat{E}}$ at $\mathcal{Q}''$ (over $\mathcal{Q}$) is equal to 
$$g \langle g_j^{I_\mathcal{Q}(t_0,t_j)} \rangle g^{-1}$$ 
for some $g \in {\rm{Gal}}(\widehat{E}_{t_0}/k_{\widehat{E}}) \subseteq \overline{G}$. Set $H={\rm{Gal}}(\widehat{E}/Ek_{\widehat{E}}).$ By \eqref{15}, the Galois group of the specialized extension $\widehat{E}_{t_0}/(Ek_{\widehat{E}})_{t_0}$ remains equal to $H$. We then obtain that the inertia group of $\widehat{E}_{t_0}/(Ek_{\widehat{E}})_{t_0}$ at $\mathcal{Q}''$ (over $\mathcal{Q}'$) is equal to $$ g \langle g_j^{I_\mathcal{Q}(t_0,t_j)} \rangle g^{-1} \bigcap H.$$ 
Hence the ramification index $e_{\mathcal{Q}'/\mathcal{Q}}$ is equal to 
\begin{equation} \label{19}
f(g) :=\frac{ \big |g \langle g_j^{I_\mathcal{Q}(t_0,t_j)} \rangle g^{-1} \big|}{ \big|g \langle g_j^{I_\mathcal{Q}(t_0,t_j)} \rangle g^{-1} \bigcap H \big|}.
\end{equation}
Conversely, let $g \in \overline{G}$. By part (2)(a) of the Specialization Inertia Theorem and as $g \in {\rm{Gal}}(\widehat{E}_{t_0}/k_{\widehat{E}})$ (\eqref{15}), $$g \langle g_j^{I_\mathcal{Q}(t_0,t_j)} \rangle g^{-1}$$ is the inertia group of some prime $\mathcal{Q}''$ lying over $\mathcal{Q}$ in $\widehat{E}_{t_0}/k_{\widehat{E}}$. Denote the restriction of $\mathcal{Q}''$ to $(Ek_{\widehat{E}})_{t_0}$ by $\mathcal{Q}'$. Then we get as before that $f(g)$ (defined in \eqref{19})
is the ramification index of $\mathcal{Q}'$ over $\mathcal{Q}$ in $(Ek_{\widehat{E}})_{t_0}/k_{\widehat{E}}$. Hence the set of all ramification indices at $\mathcal{Q}$ of $(Ek_{\widehat{E}})_{t_0}/k_{\widehat{E}}$ is equal to the set\footnote{This set does not depend on the choice of the element $g_j \in C_j$.}
$$\{f(g) : g \in \overline{G}\}.$$

Finally, consider the action of $\langle g_j^{I_\mathcal{Q}(t_0,t_j)} \rangle$ by left multiplication on the left cosets of $\overline{G}$ modulo $H$. Given $g \in \overline{G}$ and $e \in \mathbb{N}$, one has $$g_j^{e \cdot I_\mathcal{Q}(t_0,t_j)}. (gH) =gH \Longleftrightarrow g^{-1} g_j^{e \cdot I_\mathcal{Q}(t_0,t_j)} g \in H.$$ Then the orbit of the left coset $gH$ has cardinality $f(g^{-1})$. Hence the set of all ramification indices at $\mathcal{Q}$ of $(Ek_{\widehat{E}})_{t_0}/k_{\widehat{E}}$ is the set of the cardinalities of the orbits of the left cosets of $\overline{G}$ modulo $H$ under the action of $\langle g_j^{I_\mathcal{Q}(t_0,t_j)} \rangle$. As the action of $\overline{G}$ on its left cosets modulo $H$ gives the embedding morphism $\nu$, we get that the set of all ramification indices at $\mathcal{Q}$ of $(Ek_{\widehat{E}})_{t_0}/k_{\widehat{E}}$ is the set of all lengths of disjoint cycles involved in the cycle decomposition of $\nu(g_j^{I_\mathcal{Q}(t_0,t_j)})$, as needed.

\section{Specializations with specified ramified local behavior}

This section is devoted to the construction of points $t_0 \in \mathbb{P}^1(k)$ at which the specialization algebra of $E/k(T)$ consists of a single degree $n$ extension of $k$ and has specified ramified local behavior at any given finite set of prime ideals of $A$, within the limitations of the Specialization Inertia Theorem. Theorems \ref{Hilbert} and \ref{gc} give our precise  results.

\subsection{Data}

First, we state some notation for this section. Given a positive integer $s$, let $\mathcal{P}_1,\dots,\mathcal{P}_s$ be $s$ distinct good primes for $E/k(T)$.

\begin{remark}
Given $j$, there may be no $i_j$ such that $\mathcal{P}_j$ is a prime divisor of $m_{t_{i_j}}(T) \cdot m_{1/t_{i_j}}(T)$ unitizing $t_{i_j}$. In this case, if $\mathcal{P}_j$ unitizes each branch point, then, by Proposition \ref{transition2}, no specialization of $E/k(T)$ ramifies at $\mathcal{P}_j$. From now on, we assume that there exists such an index $i_j$.
\end{remark}

Let $(i_1,a_1), \dots, (i_s,a_s)$ be $s$ couples where, for each $j \in \{1,\dots,s\}$, $i_j$ is an index in $\{1,\dots,r\}$ such that $\mathcal{P}_j$ is a prime divisor of $m_{t_{i_j}}(T) \cdot m_{1/t_{i_j}}(T)$ unitizing $t_{i_j}$ and $a_j$ is a positive integer.

For $1 \leq j \leq s$, denote the set of all lengths of disjoint cycles involved in the cycle decomposition of any element of $(C_{i_j}^{S_d})^{a_j}$ by $S({i_j},a_j)$ and, given a finite $k$-\'etale algebra $\prod_{l=1}^{s'} F_l/k$, consider the following:

\vspace{1.4mm}

\noindent
(Ram/$\mathcal{P}_j$/$S({i_j},a_j)$) {\it{the following two conditions hold:

\vspace{1mm}

{\rm{(1)}} the set of all ramification indices at $\mathcal{P}_j$ of the extension $F_l/k$ is 

$S({i_j},a_j)$ for each $l \in \{1,\dots,{s'}\}$,

{\rm{(2)}} the inertia group at $\mathcal{P}_j$ of the {\it{compositum}} $\widehat{F_1} \dots {\widehat{F_{s'}}} /k$ of the Ga-

lois closures $\widehat{F_1}/k, \dots, \widehat{F_{s'}}/k$ is generated by some element of $C_{i_j}^{a_j}$.}}

\subsection{Main results} These results extend those from \cite[\S3.1]{Leg13a}.

\begin{theorem} \label{Hilbert}
Assume that $k$ is hilbertian. Then, for infinitely many points $t_0 \in k \setminus \{t_1,\dots,t_r\}$ in some arithmetic progression,

\vspace{0.5mm}

\noindent
{\rm{(1)}} the specialization algebra of $E/k(T)$ at $t_0$ consists of a single degree $n$ extension $E_{t_0}/k$ and the Galois closure $\widehat{E_{t_{0}}}/k$ has Galois group $G$,

\vspace{0.5mm}

\noindent
{\rm{(2)}} $E_{t_0}/k$ satisfies condition {\rm{(Ram/$\mathcal{P}_j$/$S({i_j},a_j)$)}} for each $j \in \{1,\dots,s\}$.

\noindent
Moreover, the fields $\widehat{E_{t_{0}}}$ may be required to be linearly disjoint over $k_{\widehat{E}}$.
\end{theorem}

Recall that a set $\Sigma$ of conjugacy classes of $\overline G$ is said to be {\it{$g$-complete}} (a terminology due to Fried \cite{Fri95}) if no proper subgroup of $\overline G$ intersects each conjugacy class in $\Sigma$. For instance, the set of all conjugacy classes of $\overline G$ is g-complete \cite{Jor72}.

Assume in Theorem \ref{gc} below that the extension $E/k(T)$ is $k$-regular and there exists a subset $I \subset \{1,\dots,r\}$ satisfying the following:

\vspace{0.5mm}

\noindent
(a) the set $\{{C_i} \, : \, i \in I \} \bigcup \, \{C_{i_j}^{a_j} \, : \, j=1,\dots,s\}$ is g-complete,

\vspace{0.5mm}

\noindent
(b) $m_{t_i}(T) \cdot m_{1/t_i}(T)$ has infinitely many prime divisors ($i \in I$).

\begin{theorem} \label{gc}
For every point $t_0 \in k \setminus \{t_1,\dots,t_r\}$ in some arithmetic progression,

\vspace{0.5mm}

\noindent
{\rm{(1)}} the specialization algebra of $E/k(T)$ at $t_0$ consists of a single degree $n$ extension $E_{t_0}/k$ and the Galois group ${\rm{Gal}}(\widehat{E_{t_{0}}}/k)$ of the Galois closure $\widehat{E_{t_{0}}}/k$ satisfies $\overline{G} \subseteq {\rm{Gal}}(\widehat{E_{t_{0}}}/k) \subseteq G$,

\vspace{0.5mm}

\noindent
{\rm{(2)}} $E_{t_0}/k$ satisfies condition {\rm{(Ram/$\mathcal{P}_j$/$S({i_j},a_j)$)}} for each $j \in \{1,\dots,s\}$.
\end{theorem}

\begin{remark}
(1) Assumption (a) holds if the set $\{C_1,\dots,C_r\}$ is itself g-complete (with $I=\{1,\dots,r\}$). Several finite extensions of $k(T)$ are known to satisfy this condition. Here is an example. 

Let $n$, $m$, $q$ and $v$ be positive integers such that $n\geq 3$, $1 \leq m \leq n$, $(m,n)=1$ and $q(n-m)-vn=1$. Then the degree $n$ $k$-regular extension of $k(T)$ generated by one root of the irreducible trinomial $Y^n-T^vY^m+T^q$ satisfies the desired condition. Indeed, this extension has branch point set $\{0, \infty, m^m n^{-n} (n-m)^{n-m}\}$, with corresponding inertia groups generated by the disjoint product of an $m$-cycle and an $(n-m)$-cycle at $0$, an $n$-cycle at $\infty$ and a transposition at $m^m n^{-n} (n-m)^{n-m}$. See \cite[$\S$2.4]{Sch00}. 

\vspace{1mm}

\noindent 
(2) Assumption (b) holds in either one of the following two situations:

\vspace{0.5mm}

\noindent
- {\hbox{each $t_i$ ($i \in I$) is $k$-rational and $A$ has infinitely many prime ideals}},

\vspace{0.5mm}

\noindent
- either $k$ is a number field or $k$ is a finite extension of a rational function field $\kappa(X)$, with $\kappa$ an arbitrary algebraically closed field of characteristic zero and $X$ an indeterminate. Indeed, this follows from the Tchebotarev density theorem in the case $k$ is a number field and this is left to the reader as an easy exercise in the function field case.

\vspace{1mm}

\noindent
(3) If $\widehat{E}/k(T)$ is itself $k$-regular, then we get $\overline{G}= {\rm{Gal}}(\widehat{E_{t_{0}}}/k) = G$ in condition (1) from Theorem \ref{gc}. This regularity condition is satisfied if either $\overline{G}=S_n$ (and then $G=S_n$ too as $E/k(T)$ has degree $n$) or $G$ is a simple group (as $\overline{G}$ is a non-trivial normal subgroup of $G$).
\end{remark}

\subsection{Proofs of Theorems \ref{Hilbert} and \ref{gc}}
The proofs rest on those of \cite[Corollaries 3.3 and 3.4]{Leg13a}, which are Theorems \ref{Hilbert} and \ref{gc} in the case the extension $E/k(T)$ is $k$-regular and Galois. We explain below how to handle the bigger generality.

\subsubsection{A central lemma}
Lemma \ref{core} below, which is proved in \cite[\S3.4]{Leg13a}, will be used on several occasions in the rest of this paper.

Denote the set of all indices $j \in \{1,\dots,s\}$ such that $t_{i_j} \not= \infty$ by $S$ and, for each $j \in \{1,\dots,s\}$, let $x_{\mathcal{P}_j} \in A$ be a generator of the maximal ideal $\mathcal{P}_j A_{\mathcal{P}_j}$ of $A_{\mathcal{P}_j}$.

\begin{lemma} \label{core}
There exists an element $\theta \in k$ that satisfies the following property. For each element $u$ of $k$ lying in $\bigcap_{l=1}^s A_{\mathcal{P}_l}$ and with $$t_{0,u}= \theta + u \cdot \prod_{l \in {S}} x_{\mathcal{P}_l}^{a_l+1},$$ one has $I_{\mathcal{P}_j}(t_{0,u}, t_{i_j})=a_j$ for each $j \in \{1,\dots,s\}$. Moreover, such an element $\theta$ may be required to lie in $A$ if $S=\{1,\dots,s\}$ (in particular, if $\infty$ is not a branch point).
\end{lemma}

The following consequence will be used on several occasions in the rest of this paper. Fix a point $t_{0,u}$ as above and assume that it is not a branch point. From Lemma \ref{3.3}, $t_{0,u}$ meets the branch point $t_{i_j}$ modulo $\mathcal{P}_j$ for each $j \in \{1,\dots,s\}$. Hence, under the condition ${\rm{Gal}}(\widehat{E}_{t_{0,u}}/k_{\widehat{E}})= \overline{G}$, we may apply part (2) of the Specialization Inertia Theorem to get that the specialization algebra of $E/k(T)$ at $t_{0,u}$ satisfies condition {\rm{(Ram/$\mathcal{P}_j$/$S({i_j},a_j)$)}} for each $j \in \{1,\dots,s\}$.

\subsubsection{Proof of Theorem \ref{Hilbert}}

Assume that $k$ is hilbertian and fix an element $\theta$ as in Lemma \ref{core}. From \cite[Lemma 3.4]{Gey78}, there exist infinitely many $u \in \bigcap_{l=1}^s A_{\mathcal{P}_l}$ such that the specializations $\widehat{E}_{t_{0,u}}/k$ of $\widehat{E}/k(T)$ at $$t_{0,u}=\theta + u \cdot \prod_{l \in {S}} x_{\mathcal{P}_l}^{a_l+1}$$ all have Galois group $G$. Hence the corresponding specialization algebras of $E/k(T)$ all consist of a single degree $n$ extension of $k$, {\it{i.e.}} part (1) of the conclusion holds with $t_0=t_{0,u}$. Moreover, for such a $t_{0,u}$, one has necessarily ${\rm{Gal}}(\widehat{E}_{t_{0,u}}/k_{\widehat{E}})= \overline{G}$. As explained in \S4.3.1, this provides part (2) of the conclusion with $t_0=t_{0,u}$, thus ending the proof of Theorem \ref{Hilbert}.

\subsubsection{Proof of Theorem \ref{gc}}

Given $i \in I$, pick a prime divisor $\mathcal{P}'_i$ of $m_{t_i}(T) \cdot m_{1/t_i}(T)$ that is a good prime for $E/k(T)$ unitizing $t_i$ (assumption (b)). We may and will require the prime ideals $\mathcal{P}'_i$ ($i \in I$) and $\mathcal{P}_1,\dots,\mathcal{P}_s$ to be distinct.

Apply Lemma \ref{core} to the larger set $\{\mathcal{P}_j \, : \, j \in \{1,\dots,s\}\} \cup \{\mathcal{P}'_i \, : \, i \in I\}$ of prime ideals, each $\mathcal{P}_j$ given with the couple $(i_j,a_j)$ from \S4.1 and each $\mathcal{P}'_i$ given with the couple $(i,1)$. Let $S'$ be the set of all indices $i \in I$ such that $t_i \not= \infty$. Then there is an element $\theta \in k$ that satisfies the following property. For each $u \in k$ satisfying $v_{\mathcal{P}_j}(u) \geq 0$ for each $j \in \{1,\dots,s\}$ and $v_{\mathcal{P}'_i}(u) \geq 0$ for each $i \in I$ and with $$t_{0,u}= \theta + u \cdot \prod_{l \in S} x_{\mathcal{P}_l}^{a_l+1} \cdot \prod_{l \in S'} x_{\mathcal{P}'_l}^{2},$$ one has $I_{\mathcal{P}_j}(t_{0,u},t_{i_j}) = a_j$ for each $j \in \{1,\dots,s\}$ and $I_{\mathcal{P}'_i}(t_{0,u},t_i) = 1$ for each $i \in I$.

Fix such a point $t_{0,u}$ and assume that it is not a branch point. By Lemma \ref{3.3}, $t_{0,u}$ meets $t_{i_j}$ modulo $\mathcal{P}_j$ for each $j \in \{1,\dots,s\}$ and $t_0$ meets $t_i$ modulo $\mathcal{P}'_i$ for each $i \in I$. From Remark \ref{proof}, the following two conditions then hold:

\noindent
- for each $j \in \{1,\dots,s\}$, there exists a prime $\mathcal{Q}_j$ of $k_{\widehat{E}}$ lying over $\mathcal{P}_j$ such that $t_{0,u}$ and $t_{i_j}$ meet modulo $\mathcal{Q}_j$ and $I_{\mathcal{Q}_j}(t_{0,u},t_{i_j})=I_{\mathcal{P}_j}(t_{0,u},t_{i_j})=a_j$,

\noindent
- for each $i \in I$, there exists a prime $\mathcal{Q}'_i$ of $k_{\widehat{E}}$ lying over $\mathcal{P}'_i$ such that $t_{0,u}$ and $t_{i}$ meet modulo $\mathcal{Q}'_i$ and $I_{\mathcal{Q}'_i}(t_{0,u},t_{i})=I_{\mathcal{P}'_i}(t_{0,u},t_{i})=1$.

\noindent
Next, apply part (2)(a) of the Specialization Inertia Theorem to the extension $\widehat{E}/k_{\widehat{E}}(T)$ and the set $\{\mathcal{Q}_j \, : \, j \in \{1,\dots,s\}\} \cup \{\mathcal{Q}'_i \, : \, i \in I\}$ of primes to obtain that ${\rm{Gal}}(\widehat{E}_{t_{0,u}}/k_{\widehat{E}})$ intersects each $C_i$ $(i \in I)$ and each $C_{i_j}^{a_j}$ ($j \in \{1,\dots,s\}$). Then we get ${\rm{Gal}}({\widehat{E}}_{t_{0,u}}/k_{\widehat{E}})=\overline{G}$ (assumption (a)). Hence condition (2) from the conclusion holds with $t_0=t_{0,u}$ (as explained in \S4.3.1). As ${\rm{Gal}}(\widehat{E}_{t_{0,u}}/k_{\widehat{E}})=\overline{G}$ and $E/k(T)$ is $k$-regular (and so the image $\nu(\overline{G})$ of $\overline{G}$ {\it{via}} $\nu : \overline{G} \rightarrow S_n$ is a transitive subgroup of $S_n$), condition (1) from the conclusion holds with $t_0=t_{0,u}$, as needed.

\section{Specializations with specified local behavior}

The aim of this section is Theorem \ref{DGL} below which combines Theorem \ref{gc} and previous work in the number field case. 

\subsection{Notation}

First, we state some notation for this section. Assume that $k$ is a number field and $A$ is the integral closure of $\Zz$ in $k$.

\subsubsection{Data for the global part} Denote the number of non-trivial conjugacy classes of $\overline{G}$ by ${\bf{cc}}(\overline{G})$. Pick ${\bf{cc}}(\overline{G})$ distinct prime numbers $p_1,\dots, p_{{\bf{cc}}(\overline{G})} \geq r^2 |\overline{G}|^2$, each of which is totally split in $k_{\widehat{E}}/\mathbb{Q}$ and such that every prime ideal of $A$ lying over of one these prime numbers is a good\footnote{Here and in \S5.1.2, condition (4) from Definition \ref{bon premier} may be removed.} prime for $E/k(T)$ \footnote{Infinitely many such primes can be found by the Tchebotarev density theorem.}. These prime numbers $p_1,\dots, p_{{\bf{cc}}(\overline{G})}$ may and will be assumed to depend only on the extension $E/\Qq(T)$.

\subsubsection{Data for the unramified part}Recall that the {\it type} of an element $\sigma \in S_d$ is the (multiplicative) divisor of all lengths of disjoint cycles involved in the cycle decomposition of $\sigma$ ({\it{e.g.}} $d$-cycles have type $d^1$).

Let ${\mathcal{S}}_{\rm{ur}}$ be a finite set of good primes for $E/k(T)$. For each $\mathcal{P} \in \mathcal{S}_{\rm{ur}}$, 

\vspace{0.5mm}

\noindent
(a) assume that the residue characteristic $p$ satisfies $p \geq r^2|\overline G|^2$ and is totally split in the extension $k_{\widehat E}/\mathbb{Q}$,

\vspace{0.5mm}

\noindent
(b) fix positive integers $d_{\mathcal{P},1}, \dots, d_{\mathcal{P},s_\mathcal{P}}$ (possibly repeated) such that $d_{\mathcal{P},1}^{1} \dots d_{\mathcal{P},s_\mathcal{P}}^{1}$ is the type of an element $\nu({g_{\mathcal{P}}})$ of $\nu(\overline{G}) \subseteq S_d$ and denote the conjugacy class of $g_\mathcal{P}$ in $G$ by $C_{\mathcal{P}}$.

\subsubsection{Data for the ramified part}
Let ${\mathcal{S}}_{\rm{ra}}$ be a finite set of good primes $\mathcal{P}$ for $E/k(T)$ such that, for each $\mathcal{P} \in {\mathcal{S}}_{\rm{ra}}$, there exists some index $i_\mathcal{P} \in \{1,\dots,r\}$ such that $t_{i_\mathcal{P}} \not= \infty$, $\mathcal{P}$ unitizes $t_{i_\mathcal{P}}$ and $\mathcal{P}$ is a prime divisor of $m_{t_{i_\mathcal{P}}}(T) \cdot m_{1/t_{i_\mathcal{P}}}(T)$. For each prime ideal $\mathcal{P} \in \mathcal{S}_{\rm{ra}}$,

\vspace{0.5mm}

\noindent
(a) assume that the ramification index and the residue degree of $\mathcal{P}$ in the extension $k/\mathbb{Q}$ both are equal to 1,

\vspace{0.5mm}

\noindent
(b) fix an integer $a_\mathcal{P} \geq 1$ and an index $i_\mathcal{P} \in \{1,\dots,r\}$ such that $t_{i_\mathcal{P}} \not= \infty$, $\mathcal{P}$ unitizes $t_{i_\mathcal{P}}$ and $\mathcal{P}$ is a prime divisor of $m_{t_{i_\mathcal{P}}}(T) \cdot m_{1/t_{i_\mathcal{P}}}(T)$.

\subsubsection{Remaining notation}

Denote the residue characteristic of a prime ideal $\mathcal{P} \in \mathcal{S}_{\rm{ur}} \cup \mathcal{S}_{\rm{ra}}$ by $p_\mathcal{P}$ and assume that the prime numbers $p _\mathcal{P}$ ($\mathcal{P} \in \mathcal{S}_{\rm{ur}} \cup \mathcal{S}_{\rm{ra}}$) and $p_l$ ($l=1, \dots, {\bf{cc}}(\overline{G})$) are distinct. Finally, set $$\beta = \prod_{l=1}^{{\bf{cc}}(\overline{G})} p_l.$$

\subsection{Statement of Theorem \ref{DGL}}

\begin{theorem} \label{DGL}
Assume that $E/k(T)$ is $k$-regular. Then, for some integer $\theta$, the following holds. For every integer ${t_0}$ such that $${t_0} \equiv \theta \, \, {\rm{mod}} \, \, \beta \cdot \prod_{\mathcal{P} \in \mathcal{S}_{\rm{ur}}} p_\mathcal{P} \cdot \prod_{\mathcal{P} \in \mathcal{S}_{\rm{ra}}} p_\mathcal{P}^{a_{{\mathcal{P}}}+1},$$ ${t_0}$ is not a branch point and the following conditions hold.

\vspace{1mm}

\noindent
{\rm{(1)}} The specialization algebra of $E/k(T)$ at $t_0$ consists of a single degree $n$ extension $E_{t_0}/k$ and the Galois group ${\rm{Gal}}(\widehat{E_{t_{0}}}/k)$ of the Galois closure $\widehat{E_{t_{0}}}/k$ satisfies $\overline{G} \subseteq {\rm{Gal}}(\widehat{E_{t_{0}}}/k) \subseteq G$.

\vspace{0.5mm}

\noindent
{\rm{(2)}} For each prime ideal $\mathcal{P} \in \mathcal{S}_{\rm{ur}}$, $\mathcal{P}$ does not ramify in $\widehat{E_{t_0}}/k$, the associated Frobenius is in the conjugacy class $C_{\mathcal{P}}$ and the integers $d_{\mathcal{P},1}, \dots, d_{\mathcal{P},s_\mathcal{P}}$ are the residue degrees at $\mathcal{P}$ of $E_{t_0}/k$.

\vspace{1mm}

\noindent
{\rm{(3)}} For each prime ideal $\mathcal{P} \in \mathcal{S}_{\rm{ra}}$, the extension $E_{t_0}/k$ satisfies condition {\rm{(Ram/$\mathcal{P}$/$S({i_\mathcal{P}},a_\mathcal{P})$)}} from \S4.1.
\end{theorem}

Theorem \ref{DGL} is proved in \S5.4.

\subsection{On the special case $k=\mathbb{Q}$} Below we assume that $k=\mathbb{Q}$ and denote the discriminant of a given finite extension $F/\mathbb{Q}$ by $d_F$.

\subsubsection{Bounds on discriminants} 

\begin{proposition} \label{DGL 2}
Let $P(T,Y) \in \mathbb{Z}[T][Y]$ be the minimal polynomial of a primitive element of $E/\mathbb{Q}(T)$, assumed to be
integral over $\mathbb{Z}[T]$. Denote the discriminant of the polynomial $P(T,Y)$
by $\Delta_P(T) \in \mathbb{Z}[T]$, the degree of $\Delta_P(T)$ by $\delta_P$ and its height\footnote{{\it{i.e.}} the maximum of the absolute values of the coefficients of $\Delta_P(T)$.} by $H(\Delta_P)$. Then, for at least one integer $t_0$ from the conclusion of Theorem \ref{DGL}, one has 
\begin{equation} \label{7}
|d_{E_{t_0}}| \leq (1+\delta_P)^{1+\delta_P} \cdot H(\Delta_P) \cdot \beta^{\delta_P} \cdot \Big(\prod_{p \in \mathcal{S}_{\rm{ur}}} p \cdot \prod_{p \in \mathcal{S}_{\rm{ra}}} p^{a_p + 1} \Big)^{\delta_P}.
\end{equation}
\end{proposition}

Proposition \ref{DGL 2} is proved in \S5.3.3.

\begin{remark} \label{DGL 3}
(1) Some lower bounds can also be given. Indeed, assume for example that, for each prime number $p \in \mathcal{S}_{\rm{ra}}$, the integer $a_p$ is not a multiple of the order of the elements of $C_{i_p}$. Then, for every specialization point $t_0$ from the conclusion of Theorem \ref{DGL}, each prime number $p \in \mathcal{S}_{\rm{ra}}$ ramifies in the specialization $E_{t_0}/\mathbb{Q}$. This provides $$\prod_{p \in \mathcal{S}_{\rm{ra}}} p \leq |d_{E_{t_0}}|.$$

In particular, we obtain some extra limitations on the ramification in our specializations. Namely, consider a specialization $E_{t_0}/\mathbb{Q}$ as in Proposition \ref{DGL 2} and denote the set of all prime numbers $p \not \in \mathcal{S}_{\rm{ra}}$ that ramify in $E_{t_0}/\mathbb{Q}$ by $\mathcal{S}'_{\rm{ra}}$. Then one has 
\begin{equation} \label{4}
\prod_{p \in \mathcal{S}_{\rm{ra}} \cup \mathcal{S}'_{\rm{ra}}} p \leq |d_{E_{t_0}}|.
\end{equation}
Combining \eqref{4} and \eqref{7} then provides
$$\prod_{p \in \mathcal{S}'_{\rm{ra}}} p \leq (1+\delta_P)^{1+\delta_P} \cdot H(\Delta_P) \cdot \beta^{\delta_P} \cdot \Big(\prod_{p \in \mathcal{S}_{\rm{ur}}} p^{\delta_P} \Big) \cdot \Big(\prod_{p \in \mathcal{S}_{\rm{ra}}} p^{a_p \delta_P + \delta_P-1}\Big).$$

Moreover, if the extension $E_{t_0}/\mathbb{Q}$ is Galois (in particular, if $E/\mathbb{Q}(T)$ is itself Galois), then \cite[\S1.4, Proposition 6]{Ser81} may be used to replace \eqref{4} by the better inequality $$\prod_{p \in \mathcal{S}_{\rm{ra}} \cup \mathcal{S}'_{\rm{ra}}} p ^{n/2} \leq |d_{E_{t_0}}|.$$ 

\noindent
(2) Similar bounds on the discriminant of the Galois closure $\widehat{E_{t_0}}/\mathbb{Q}$ can also be given as \cite[\S1]{Ser81} provides
$$|d_{E_{t_0}}| \leq |d_{\widehat{E_{t_0}}}| \leq |G|^{|G|} \cdot |d_{E_{t_0}}|^{|G| + |G|^2/2}.$$

\noindent
(3) The above bounds are in a sense the best possible as the following evidence suggests.

Assume that some branch point $t_i$ of the extension $E/\mathbb{Q}(T)$ is in $\Qq$. Let $m_0$ be a real number satisfying $m_0 > p_l$ for each $l \in \{1,\dots, {\bf{cc}}(\overline{G})\}$ and such that each prime number $p\geq m_0$ is a good prime for $E/\mathbb{Q}(T)$ unitizing $t_i$. As the prime numbers $p_1,\dots, p_{{\bf{cc}}(\overline{G})}$ depend only on the extension $E/\Qq(T)$, we may and will assume that the same holds for $m_0$. Given a real number $x \geq m_0$, apply Theorem \ref{DGL} with $\mathcal{S}_{\rm{ur}} = \emptyset$, $\mathcal{S}_{\rm{ra}}$ taken to be the set of all prime numbers $p \in [m_0,x]$ (this can be done as every prime number $p \geq m_0$ is a prime divisor of $m_{t_i}(T) \cdot m_{1/t_i}(T)$) and $a_p=1$ for each $p \in \mathcal{S}_{\rm{ra}}$. We then obtain an extension $E_{t_{0,x}}/\mathbb{Q}$ ramifying at each prime number $p \in [m_0,x]$ (and which also satisfies the remaining properties from the conclusion) and which by the above may be assumed to satisfy $$\prod_{p \in [m_0,x]} p \leq |d_{E_{t_{0,x}}}| \leq \gamma \cdot \prod_{p \in [m_0,x]} p^{2\delta}$$ with 
$\gamma = (1+\delta_P)^{1+\delta_P} \cdot H(\Delta_P) \cdot \beta^{\delta_P}$ and $\delta=\delta_P$. Hence one has 
$$\alpha \cdot \prod_{p \leq x} p \leq |d_{E_{t_{0,x}}}| \leq \gamma \cdot \prod_{p \leq x} p ^{2 \delta}$$ for some positive constants $\alpha$, $\gamma$ and $\delta$ depending only on $E/\mathbb{Q}(T)$. As $${\rm{log}} \, \Big(\prod_{p \leq x} p \Big) \,  \sim \, x ,\quad x \rightarrow \infty,$$ we get $$c_1 \, x \leq {\rm{log}} \, |d_{E_{t_{0,x}}}| \leq c_2 \, x$$ for some positive constants $c_1$ and $c_2$ depending only on $E/\mathbb{Q}(T)$ (and sufficiently large $x$).
\end{remark}

\subsubsection{Proof of Theorem 2}

Now, we explain how to obtain Theorem 2 from the introduction. Assume that the extension $E/\mathbb{Q}(T)$ has at least one $\Qq$-rational branch point $t_i$ and the Galois closure $\widehat{E}/\mathbb{Q}(T)$ is $\Qq$-regular. Up to applying a suitable change of variable, we may assume that $t_i \not=\infty$. Then all but finitely many prime numbers are prime divisors of $m_{t_i}(T) \cdot m_{1/t_i}(T)$ (as $t_i \in \Qq$) and condition (a) from \S5.1.2 only requires $p$ to be sufficiently large (as $\Qq_{\widehat{E}}=\Qq$). Then conditions (1), (2) and (3) in the conclusion of Theorem 2 follow from Theorem \ref{DGL} (applied with $d_{p,1}^1 \dots d_{p,s_p}^1=1^d$ for each prime number $p \in \mathcal{S}_{\rm{ur}}$ and $(a_p,i_p)=(1,i)$ for each prime number $p \in \mathcal{S}_{\rm{ra}}$) and the fact that $G=\overline{G}$. As to condition (4), it is a consequence of the bounds given in Proposition \ref{DGL 2} and in part (1) of Remark \ref{DGL 3}.

\subsubsection{Proof of Proposition \ref{DGL 2}}
Pick an integer $u \in [0, \delta_P]$ such that $$t_{0} = \theta + u \cdot \beta \cdot \prod_{p \in \mathcal{S}_{\rm{ur}}} p \cdot \prod_{p \in \mathcal{S}_{\rm{ra}}} p^{a_p+1}$$ is not a root of $\Delta_P(T)$ (with $\theta$ as in Theorem \ref{DGL}). As we may assume 
$$1 \leq \theta \leq \beta \cdot \prod_{p \in \mathcal{S}_{\rm{ur}}} p \cdot \prod_{p \in \mathcal{S}_{\rm{ra}}} p^{a_p+1},$$ one has 
\begin{equation} \label{5}
1 \leq |t_{0}| \leq (1+\delta_P) \cdot \beta \cdot \prod_{p \in \mathcal{S}_{\rm{ur}}} p \cdot \prod_{p \in \mathcal{S}_{\rm{ra}}} p^{a_p+1}.
\end{equation}
From condition (1) in the conclusion of Theorem \ref{DGL} and as $\Delta_P(t_0) \not=0$, the polynomial $P(t_0,Y)$ is irreducible over $\mathbb{Q}$ (Lemma \ref{spec non gal}). As it is monic and has coefficients in $\mathbb{Z}$, its discriminant $\Delta_P(t_0)$ is a multiple of $d_{E_{t_0}}$. Hence one has 
\begin{equation} \label{6}
|d_{E_{t_0}}| \leq |\Delta_P(t_0)| \leq (1+ \delta_P) \cdot H(\Delta_P) \cdot |t_0|^{\delta_P}.
\end{equation}
(as $|t_0| \geq 1$). It then remains to combine \eqref{5} and \eqref{6} to get \eqref{7}, thus ending the proof.

\subsection{Proof of Theorem \ref{DGL}}
First, we recall how \cite{DL12} handles condition (2) from the conclusion. Given a prime ideal $\mathcal{P} \in \mathcal{S}_{\rm{ur}}$, denote the order of $g_\mathcal{P}$ by $e_\mathcal{P}$ and, for simplicity, denote the residue characteristic of $\mathcal{P}$ by $p$. Let $F^{p,e_\mathcal{P}}/\mathbb{Q}_p$ be the unique unramified Galois extension of $\mathbb{Q}_p$ of degree $e_\mathcal{P}$, given with an isomorphism ${\rm{Gal}}(F^{p,e_\mathcal{P}}/\mathbb{Q}_p) \rightarrow \, \langle g_\mathcal{P} \rangle$ mapping the Frobenius of the extension $F^{p,e_\mathcal{P}}/\mathbb{Q}_p$ to $g_\mathcal{P}$. Let $\varphi: {\rm{G}}_{\mathbb{Q}_p} \rightarrow \,  \langle g_\mathcal{P} \rangle$ be the corresponding epimorphism (with ${\rm{G}}_{\mathbb{Q}_p}$ the absolute Galois group of $\Qq_p$). As $p \geq r^2|\overline{G}|^2$ and $\mathcal{P}$ is a good prime for $E/k(T)$, we may use \cite{DL12} to get that there exists an integer ${\theta_\mathcal{P}}$ that satisfies the following property. For every integer $t$ satisfying $t \equiv {\theta_\mathcal{P}} \, \, {\rm{mod}} \, \, p$, $t$ is not a branch point and 

\noindent
- the specialization of $\widehat{E}\mathbb{Q}_p/\mathbb{Q}_p(T)$ at $t$ corresponds to the epimorphism $\varphi$ (note that $k_{\widehat{E}}\mathbb{Q}_p = \mathbb{Q}_p$ as $p$ has been assumed to be totally split in the extension $k_{\widehat{E}}/\mathbb{Q}$),

\noindent
- the specialization algebra of ${E}\mathbb{Q}_p/\mathbb{Q}_p(T)$ at $t$ is equal to $\prod_{l=1}^{s_\mathcal{P}} F^{p,d_{\mathcal{P},l}}/ \mathbb{Q}_p$ where $F^{p,d_{\mathcal{P},l}}/ \mathbb{Q}_p$ denotes the unique unramified extension of $\mathbb{Q}_p$ of degree $d_{\mathcal{P}, l}$.

Given $l \in \{1, \dots, {\bf{cc}}(\overline{G})\}$, pick a prime ideal $\mathcal{P}_l$ of $A$ lying over $p_l$ in the extension $k/\Qq$ and associate a non-trivial conjugacy class $C_l$ of $\overline{G}$. We may and will assume that the map $l \mapsto C_l$ is a bijection between the set $\{1, \dots, {\bf{cc}}(\overline{G})\}$ and the set of all non-trivial conjugacy classes of $\overline{G}$. For each $l$, \cite{DL12} provides as above an integer ${\theta_l}$ that satisfies the following. For each integer $t$ satisfying $t \equiv {\theta_l} \, \, {\rm{mod}} \, \, p_l$, $t$ is not a branch point and the Galois group of the specialization of $\widehat{E}\mathbb{Q}_{p_l}/\mathbb{Q}_{p_l}(T)$ at $t$ is conjugate in $\overline{G}$ to an element of $C_l$.

Given a prime ideal $\mathcal{P} \in {\mathcal{S}}_{\rm{ra}}$, denote the residue characteristic by $p$. From Lemma \ref{core}, the fact that $t_{i_\mathcal{P}} \not= \infty$ and as $\mathcal{P}$ is unramified in the extension $k/\Qq$ (condition (a) from \S5.1.3), there exists an element ${\theta'_\mathcal{P}} \in A$ such that 
\begin{equation} \label{28}
I_{\mathcal{P}}(\theta'_\mathcal{P} + u \cdot p^{a_\mathcal{P}+1}, t_{i_\mathcal{P}})=a_\mathcal{P}
\end{equation}
for every $u \in A_\mathcal{P}$. We claim that $\theta'_\mathcal{P}$ may be chosen in $\Zz$. Indeed, by the full condition (a) from \S5.1.3, the completion of $k$ with respect to $\mathcal{P}$ is equal to $\Qq_p$. Viewing $\theta'_\mathcal{P}$ as an element of $\mathbb{Z}_p$, pick an integer $\theta''_\mathcal{P}$ such that 
\begin{equation} \label{29}
v_\mathcal{P}(\theta'_\mathcal{P} - \theta''_\mathcal{P}) \geq a_\mathcal{P}+1.
\end{equation}
Let $u \in A_\mathcal{P}$. As $v_\mathcal{P}(\theta'_\mathcal{P} + u \cdot p^{a_\mathcal{P}+1}) \geq 0$, one has
\begin{equation} \label{30}
I_{\mathcal{P}}(\theta'_\mathcal{P} + u \cdot p^{a_\mathcal{P}+1}, t_{i_\mathcal{P}}) = v_\mathcal{P}(m_{t_{i_\mathcal{P}}}(\theta'_\mathcal{P} + u \cdot p^{a_\mathcal{P}+1})).
\end{equation}
By \eqref{28} and \eqref{30}, we get
\begin{equation} \label{31}
v_\mathcal{P}(m_{t_{i_\mathcal{P}}}(\theta'_\mathcal{P} + u \cdot p^{a_\mathcal{P}+1})) =a_\mathcal{P}.
\end{equation}
Then combining \eqref{29} and \eqref{31} provides $$v_\mathcal{P}(m_{t_{i_\mathcal{P}}}(\theta''_\mathcal{P} + u \cdot p^{a_\mathcal{P}+1})) =a_\mathcal{P},$$
{\it{i.e.}} $$I_{\mathcal{P}}(\theta''_\mathcal{P} + u \cdot p^{a_\mathcal{P}+1}, t_{i_\mathcal{P}})=a_\mathcal{P},$$ (as $v_\mathcal{P}(\theta''_\mathcal{P} + u \cdot p^{a_\mathcal{P}+1}) \geq 0$), thus proving our claim. From now on, we assume that $\theta'_\mathcal{P}$ lies in $\mathbb{Z}$. In particular, one has $I_{\mathcal{P}}(t, t_{i_\mathcal{P}})=a_\mathcal{P}$ for every integer $t$ satisfying $t \equiv \theta'_\mathcal{P} \, \, {\rm{mod}} \, \, p^{a_\mathcal{P}+1}$.

Next, use the Chinese remainder theorem to find some integer $\theta$ that satisfies the following three conditions:

\noindent
- $\theta \equiv \theta_\mathcal{P} \,  \, {\rm{mod}} \, \, p_\mathcal{P}$ for each prime ideal $\mathcal{P} \in \mathcal{S}_{\rm{ur}}$, 

\noindent
- $\theta \equiv \theta_l \,  \, {\rm{mod}} \, \, p_l$ for each $l \in \{1, \dots, {\bf{cc}}(\overline{G})\}$,

\noindent
- $\theta\equiv {\theta'_\mathcal{P}} \, \,   {\rm{mod}} \, \, p_\mathcal{P}^{a_\mathcal{P}+1}$ for each prime ideal $\mathcal{P} \in \mathcal{S}_{\rm{ra}}$. 

\noindent
Let $t_0$ be an integer that satisfies $$t_0 \equiv \theta \,  \, {\rm{mod}} \, \, \beta \cdot \prod_{\mathcal{P} \in \mathcal{S}_{\rm{ur}}} p_\mathcal{P} \cdot \prod_{\mathcal{P} \in \mathcal{S}_{\rm{ra}}} p_\mathcal{P}^{a_\mathcal{P}+1}.$$ 
In  particular, $t_0$ is not a branch point of the extension $E/k(T)$ (as ${\bf{cc}}(\overline{G}) \geq 1$). By the conclusion on the prime ideals $\mathcal{P}_1, \dots, \mathcal{P}_{{\bf{cc}}(\overline{G})}$ and \cite{Jor72}, one has ${\rm{Gal}}(\widehat{E}_{t_0}/k_{\widehat{E}})=\overline{G}$. As the extension $E/k(T)$ has been assumed to be $k$-regular, its specialization algebra at $t_0$ then consists of a single degree $n$ extension $E_{t_0}/k$. Hence condition (1) in the conclusion holds. Condition (2) follows from the above conclusion of the primes in $\mathcal{S}_{{\rm{ur}}}$. Finally, one has $I_{\mathcal{P}}(t_0, t_{i_\mathcal{P}}) = a_\mathcal{P}$ for each prime ideal $\mathcal{P} \in \mathcal{S}_{\rm{ra}}$. Combining this and the condition ${\rm{Gal}}(\widehat{E}_{t_0}/k_{\widehat{E}})=\overline{G}$ provides condition (3) (as explained in \S4.3.1), thus ending the proof.

\bibliography{Biblio2}
\bibliographystyle{alpha}

\end{document}